\documentclass[titlepage,a4paper,11pt]{amsart} 
\usepackage[foot]{amsaddr}
\usepackage{amssymb}
\usepackage[lite]{amsrefs}
\usepackage{t1enc}
\usepackage[latin1]{inputenc}
\usepackage[english]{babel}
\usepackage{mathrsfs} 
\usepackage{hyperref}
\usepackage[all]{xy} 
\usepackage[lite]{amsrefs}
\usepackage[left=3cm,top=3cm,right=3cm,bottom=3cm]{geometry}

\newtheorem{thm}{Theorem}[section]
\newtheorem{prop}[thm]{Proposition}
\newtheorem{lem}[thm]{Lemma}
\newtheorem{cor}[thm]{Corollary}
\newtheorem*{defn}{Definition}
\newtheorem{rem}[thm]{Remark}
\newtheorem{ex}[thm]{Example}

\DeclareMathOperator\Isom{Isom}
\DeclareMathOperator\tr{tr}

\DeclareMathOperator\h{\mathcal H}
\DeclareMathOperator\B{\mathcal B}

\DeclareMathOperator\oo{{\mathcal O}}

\DeclareMathOperator\N{\mathbb{N}}
\DeclareMathOperator\Z{\mathbb{Z}}

\DeclareMathOperator\id{id}

\hypersetup{
  colorlinks   = true, 
  urlcolor     = blue, 
  linkcolor    = blue, 
  citecolor   = red    
}
\title{On Schatten restricted norms}
\date{\today}
\author{Martin Miglioli}
\email[Martin Miglioli]{martin.miglioli@gmail.com}
\address[Martin Miglioli]{Instituto Argentino de Matem\'atica-CONICET. Saavedra 15, Piso 3, (1083) Buenos Aires, Argentina}
\thanks{The author was supported by IAM-CONICET, grants PIP 2010-0757 (CONICET) and PICT 2010-2478 (ANPCyT)}

\begin{document}
\begin{abstract}
We consider norms on a complex separable Hilbert space such that $\langle a\xi,\xi\rangle\leq \|\xi\|^2\leq\langle b\xi,\xi\rangle$ for positive invertible operators $a$ and $b$ that differ by an operator in the Schatten class. We prove that these norms have unitarizable isometry groups, our proof uses a generalization of a fixed point theorem for isometric actions on positive invertible operators. As a result, if the isometry group does not leave any finite dimensional subspace invariant, then the norm must be Hilbertian. That is, if a Hilbertian norm is changed to a close non-Hilbertian norm, then the isometry group does leave a finite dimensional subspace invariant. The approach involves metric  geometric arguments related to the canonical action of the group on the non-positively curved space of positive invertible Schatten perturbations of the identity .\\

\medskip

\end{abstract}

\maketitle
\tableofcontents




\section{Introduction}
This article discusses some rigidity aspects related to Mazur's Rotation Problem; which asks whether any separable transitive Banach space (that is, a Banach space where any point on the unit sphere can be mapped into any other point on the unit sphere by surjective isometry) is necessarily isometric to a Hilbert space. See the classic work \cite{banach}, or the recent paper \cite{fr} and the references therein for the relation between properties of norms and its isometry groups. We show here that if a norm is close in a certain sense to a Hilbert space norm then its isometry group is unitarizable. And we also show that if the isometry group of a norm of this type is irreducible, then it must be a Hilbert space norm. 

The approach in this article is metric geometric: the space of Hilbert norms which are close to the standard Hilbert norm has the structure of a non-positively curved space. A change of variables acts isometrically on this space and sharp conditions on groups of isometries are usually not required. We will relate order properties of the norms with geometric quantities, like the diameter of orbits of positive invertible operators.

The article is organized as follows. In Section \ref{prel} we review some facts on the metric geometry of positive invertible operators which are Schatten perturbations of the identity, and we recall properties of the change of variables action on the norms. Some new results about larger groups acting on these positive invertible operators are proved. In Section \ref{s1} we prove a generalization of a fixed point theorem for isometric actions on these spaces of positive invertible operators and apply it to the unitarization of the isometry groups of some norms. In Section \ref{s2} we prove metric properties of two natural invariant sets associated to a norm. Finally, in Section \ref{s3} we relate the irreducibility of the isometry group to a unique fixed point property.

\section{Preliminaries}\label{prel}

\subsection{Geometry on the space $P_p$}\label{algtrace}

In this section, we recall some geometric facts about the space $P_p$ of positive invertible operators acting on a complex separable Hilbert space $\h$ which are $p$-Schatten perturbations of the identity for some $1<p<\infty$. Throughout this article, we will mainly use the metric and the geodesic structure of $P_p$; we mention the differential geometric background for the sake of completeness. See \cite{conde}, where the geometry of this space was first studied, and Section 3 in \cite{condelarotonda2} for further information on $p$-Busemann spaces.

Let $\B_p(\h)$ stand for the bilateral ideal of $p$-Schatten operators of $\B(\h)$. Recall that $\B_p(\h)$ is a Banach algebra consisting of all compact operators such that the Schatten norm defined by 
$$\|a\|_p=\tr(|a|^p)^{\frac{1}{p}}$$
is finite. Then, $\|a\|\leq \|a\|_p$ for $a\in \B_p(\h)$. 

The space of positive invertible $p$-Schatten perturbations of the identity operator is the space 
$$P_p=P\cap(\B_p(\h)+\id)=\exp(\B_p(\h)_s),$$
where $P$ are the positive invertible operators acting on a complex separable Hilbert space $\h$ and $B_p(\h)_s$ are the self-adjoint operators in $\B_p(\h)$.  
 
The space $P_p$ carries a canonical symmetric space structure with Cartan symmetries given by $\sigma_a(b)=ab^{-1}a$ 
for $a,b\in P_p$, see \cite{neeb}*{Example 3.9}. The corresponding exponential map $\exp_{\id}:T_{\id}P_p\simeq \B_p(\h)_s\to P_p$ at the identity element $\id\in P$ is given by the ordinary exponential with inverse $\log:P_p\to \B_p(\h)_s$ and the corresponding geodesics between any two points $a,b\in P_p$ are given by
\begin{align*}
\gamma_{a,b}(t)=a^{\frac12}(a^{-\frac12}ba^{-\frac12})^ta^{\frac12}.
\end{align*}
For $a\in P_p$, one can identify the tangent space $T_a(P_p)$ of $P_p$ at $a$ with the space $\B_p(\h)_s$, and endow this manifold with a Finsler metric by means of the following formula: $\|X\|_a=\|a^{-\frac12}Xa^{-\frac12}_2\|_p$. The geodesics between any two points $a,b\in P$ are length minimizing, and they are unique as shortest continuous paths between $a$ and $b$ by the uniform convexity of the Schatten norms, that is 
$$d_p(a,b)=\mathrm{Length}(\gamma_{a,b})=\left\Vert\log\left(a^{-\tfrac12}ba^{-\tfrac12}\right)\right\Vert_p.$$

The metric space $(P_p,d_p)$ is complete and satisfies a generalized semi-parallelogram law, hence, it has semi-negative curvature. More precisely, it is a $p$-Busemann space, that is, for $a\in P_p$ and a geodesic $\gamma:[0,1]\to P_p$ we have  
$$d_p(a,\gamma(\frac{1}{2}))^r\leq\frac{1}{2}(d_p(a,\gamma(0))^r+d_p(a,\gamma(1))^r)-\frac{1}{4}c_r(\gamma(0),\gamma(1))^r,$$
where $r=\max\{p,2\}$ and $c_r=p-1$ if $r=2$ or $c_r=\frac{1}{2^{p-2}}$ if $r\neq 2$ (see Lemma 3.4 and Theorem 3.5 in \cite{conde}). This metric property implies that projections to closed geodesically convex sets, and circumcenters of bounded sets, are well defined.

We denote the invertible operators in $\B_p(\h)+ \id$ by $G_p$. The action of this group on $P_p$ given by 
$$g\cdot a=(g^{-1})^*ag^{-1}$$ 
is isometric and the isotropy group of the identity $\id$ in $P_p$ is $U_p=U\cap \B_p(\h)$, where $U\subseteq\B(\h)$ is the unitary group. Hence $P_p=G_p/U_p$. 

We add an observation which is relevant in the context of this article: the whole unitary group acts isometrically on $P_p$, and we can therefore enlarge the isometry group $G_p$ to the extended group $G_{p,ext}$ which includes all the unitaries.

\begin{defn}
The extended group $G_{p,ext}$ is the product $UP_p$, where $U$ are the unitary operators of $\B(\h)$.
\end{defn} 

\begin{prop}\label{chargrupoext}
The set $G_{p,ext}$ is a subgroup of the invertible operators in $\B(\h)$. It can be characterized as the set of invertible operators $g$ such that
$$g^*g-\id\in \B_p(\h).$$
\end{prop}

\begin{proof}
If $g=up$ is the polar decomposition of an element of $G_{p,ext}$, then $g^{-1}=p^{-1}u^{-1}=u^{-1}(up^{-1}u^{-1})$ and $up^{-1}u^{-1}$ is in $P_p$, hence $g^{-1}\in G_{p,ext}$. 

If $g_1=u_1p_1$ and $g_2=u_2p_2$ are elements of $G_{p,ext}$, then 
$$g_1g_2=u_1p_1u_2p_2=u_1u_2(u_2^{-1}p_1u_2)p_2.$$
Since $(u_2^{-1}p_1u_2)p_2$ is a product of two operators in $G_p$ it is in $G_p$ and has therefore a polar decomposition $(u_2^{-1}p_1u_2)p_2=up$ with $u\in U_p$ and $p\in P_p$. Hence $g_1g_2=u_1u_2up\in G_{p,ext}$.

Let $g$ be an invertible operator and $g=up$ its polar decomposition, then $g^*g=p^2$. Hence $g^*g-\id\in \B_p(\h)$ if and only $p^2-\id\in \B_p(\h)$. And $p^2-\id\in \B_p(\h)$ if and only if $p^2\in P_p$, which is equivalent to the condition $p\in P_p$. The second claim of the lemma follows.
\end{proof}

\begin{prop}\label{invariancedist}
The distance function $d_p(a,b)=\mathrm{Length}(\gamma_{a,b})=\|\log(a^{-\frac12}ba^{-\frac12})\|_p$ is invariant for the action of the full unitary group $U\subseteq \B(\h)$.
\end{prop}
\begin{proof}  
Note that by equivariance under unitary conjugation of the functional calculus, and by the unitary invariance of the Schatten norm
\begin{align*}
d_p(uau^{-1},ubu^{-1})&=\|\log((uau^{-1})^{-\frac12}(ubu^{-1})(uau^{-1})^{-\frac12})\|_p\\
&=\|\log(ua^{-\frac12}ba^{-\frac12}u^{-1})\|_p\\
&=\|u\log(a^{-\frac12}ba^{-\frac12})u^{-1}\|_p=d_p(a,b).
\end{align*}
\end{proof}

Propositions \ref{chargrupoext}, \ref{invariancedist} and the fact that $G_p=U_pP_p$ acts isometrically on $P_p$ imply the following: 
 
\begin{prop}
The extended group $G_{p,ext}=UP_p$ acts transitively and isometrically on $P_p$ and has $U$ as the isotropy at the identity. Hence we can write 
$$P_p=G_{p,ext}/U.$$
\end{prop}

The following will not be used in this article.

\begin{rem}
If $(\xi_n)_n$ is an orthonormal basis of $\h$ and $a,b\in P_p$, then there exist $g\in G_{p,ext}$ such that $g\cdot a$ and $g\cdot b$ are diagonal with respect to the orthonormal basis: first apply $a^{\frac{1}{2}}\cdot a=\id$ and $a^{\frac{1}{2}}\cdot b=c$, then we pick an $u\in U$ such that $u\cdot c=d$ is diagonal with respect to $(\xi_n)_n$. Then $g=ua^{\frac{1}{2}}$ does the job.
\end{rem}

\subsection{Basic properties of the canonical action on norms}\label{basicprop}

We identify positive invertible operators and Hilbert norms by means of $\|\xi\|_a=\langle a\xi,\xi\rangle^{\frac{1}{2}}$. The group of invertible operators acts as a change of variables on the set of norms compatible with the Hilbert norm as 
$$g\mapsto (\|\cdot\|\mapsto \|g^{-1}\cdot\|).$$
If we apply this change of variables to Hilbert norms then we obtain
\begin{align}\label{changeofvariablesaction}
g\mapsto (\|\cdot\|_a\mapsto \|g^{-1}\cdot\|_a=\|\cdot\|_{(g^{-1})^*ag^{-1}}),
\end{align} 
so this action is identified with the canonical action $g\cdot a=(g^{-1})^*ag^{-1}$ on the positive invertible operators.

\begin{defn}
The isometry group of a norm $\|\cdot\|$ is
$$\Isom(\|\cdot\|)=\{h\in\B(\h):h \mbox{  is invertible and  }\|h\xi\|=\|\xi\|\mbox{  for all  }\xi\in\h\},$$
and it consists of the change of variables which leave the norm fixed. 
\end{defn}

We have the following identity 
$$\Isom(\|g^{-1}\cdot\|)=g\Isom(\|\cdot\|)g^{-1}.$$

The set of norms has a natural order structure and we write $\|\cdot\|\leq\|\cdot\|'$ if $\|\xi\|\leq\|\xi\|'$ for all $\xi\in\h$. For Hilbertian norms $\|\cdot\|_a\leq\|\cdot\|_b$ if and only if $a\leq b$. This order structure is preserved under change of variables.

\begin{rem}
In most articles on the geometry of the positive invertible operators the canonical isometric action by invertible elements on the cone is given by $g\cdot a=gag^*$. The change of variables by an invertible $g$ is as follows
$$\|g^{-1}\xi\|_a=\langle ag^{-1}\xi,g^{-1}\xi\rangle^{\frac{1}{2}}=\langle (g^{-1})^*ag^{-1}\xi,\xi\rangle^{\frac{1}{2}}.$$
So the action is given by $g\mapsto (a\mapsto (g^{-1})^*ag^{-1})$. But $g\mapsto (g^{-1})^*$ is an automorphism, hence composing with this automorphism we get the action
$$g\mapsto (g^{-1})^*\mapsto (a\mapsto (g^{-1})^*ag^{-1})$$
considered in this article. For simplicity in the notation, instead of taking a group element $g$ acting on the space $P$, we will often take $g^{-1}$.
\end{rem}

We recall from Section 2.2 in \cite{ms} the basic non-metric properties of the canonical action restricted to subgroups $H\subseteq G$. Let a group $G$ act on a set $X$ and let $H$ be a subgroup of $G$. The fixed point set for this action is denoted by $X^H$. The orbit of $a\in P$ shall be denoted by $\oo_H(a)$.
We next recall how orbits and fixed point sets behave under translations; for $f\in G$ and $x\in X$ we have 
$$f^{-1}\cdot\oo_H(x)=\oo_{f^{-1}Hf}(f^{-1}\cdot x)$$
and 
$$f^{-1}\cdot X^H=X^{f^{-1}Hf}.$$

The group $G_{p,ext}$ acts on $P_p$ as $g\cdot a=(g^{-1})^*ag^{-1}$.  A subgroup $H$ is said to be \textit{unitarizable}, if there is an invertible operator $s$ such that $s^{-1}Hs$ is a group of unitaries. Note that if $s$ is a unitarizer of $H$ and $s=bu$ is its polar decomposition, then $b$ is a positive unitarizer of $H$ as $u^{-1}b^{-1}Hbu$ is a group of unitaries. The next well known proposition relates positive unitarizers to fixed points of the action, we include a proof for the benefit of the reader.

\begin{prop}\label{fixed}
An operator $s\in P_p$ unitarizes $H\subseteq G_{p,ext}$, that is, $s^{-1}Hs$ is a group of unitaries, if and only if $s^{-2}$ is a fixed point of the canonical action of $H$ on $P_p$. 
\end{prop}

\begin{proof}
Observe that
\begin{align*}
s^{-1}Hs\subseteq U &\Leftrightarrow s^{-1}hs(s^{-1}hs)^*=\id \mbox{  for all  } h\in H\\
&\Leftrightarrow s^{-1}hs^2h^*s^{-1}=\id \mbox{  for all  } h\in H \\
&\Leftrightarrow hs^2h^*=s^2 \mbox{  for all  } h\in H\\
&\Leftrightarrow (h^{-1})^*s^{-2}h^{-1}=s^{-2} \mbox{  for all  } h\in H,
\end{align*}
where in the last equivalence we applied inversion on both sides. 
\bigskip
\end{proof}

\section{Main results}\label{main}

\subsection{Unitarization of the group of isometries}\label{s1}

\begin{defn}
We say that a norm $\|\cdot\|$ is Schatten close to a Hilbert norm if for positive invertible operators $a$ and $b$ such that $a-b$ is in a Schatten class
$$\|\cdot\|_a\leq \|\cdot\|\leq\|\cdot\|_b.$$
If instead of requiring $a-b\in\B_p(\h)$ we require the stronger condition $a,b\in P_p$ then we say that the norm $\|\cdot\|$ is Schatten close to the standard Hilbert norm. 
\end{defn}

\begin{lem}\label{changenorm}
If a norm is Schatten close to a Hilbert norm, then a change of variables can be made so that the new norm satisfies $\|\cdot\|_{\id}\leq \|\cdot\|\leq\|\cdot\|_{c}$ for $c\in P_p$, hence it is Schatten close to the standard Hilbert norm.
\end{lem}

\begin{proof}
We can apply a change of variables to $\|\cdot\|_a\leq \|\cdot\|\leq\|\cdot\|_b$ and by (\ref{changeofvariablesaction}) we get
$$\|\cdot\|_{\id}\leq \|a^{-\frac{1}{2}}\cdot\|\leq\|\cdot\|_{a^{-\frac{1}{2}}ba^{-\frac{1}{2}}}.$$
We denote $c=a^{-\frac{1}{2}}ba^{-\frac{1}{2}}$, and observe that $c\geq \id$ and $c-\id\in\B_p(\h)$ because $a\leq b$ and $a-b=s\in \B_p(\h)$, since $\id-a^{-\frac{1}{2}}ba^{-\frac{1}{2}}=a^{-\frac{1}{2}}sa^{-\frac{1}{2}}\in \B_p(\h)$.
\end{proof}

\begin{rem}\label{remunitar}
A change of variables does not alter the unitarizability of subgroups. Hence, by a change of variables we can assume that the norms are Schatten close to the standard Hilbert norm. We will do this for the rest of the article.
\end{rem}

\begin{lem}\label{boundedorbit}
If $H$ is a subgroup of $G_{p,ext}$, then
$$\sup_{h\in H}\|h^*h-\id\|_p< \infty\mbox{  if and only if  }\sup_{h\in H}\|\log(h^*h)\|_p< \infty.$$ 
\end{lem}

\begin{proof}
First, $\sup_{h\in H}\|h^*h-\id\|_p=C < \infty$ implies
$$\|h^*h\|-\|\id\|\leq \|h^*h-\id\|\leq \|h^*h-\id\|_p\leq C\mbox{  for all  } h\in H,$$ 
so that $h^*h \leq (C+1)\id$. If we take the inverse map then $(C+1)^{-1}\id \leq (h^{*}h)^{-1}=h^{-1}(h^{-1})^*$. Taking $h^{-1}$ instead of $h$ we have $(C+1)^{-1}\id \leq hh^{*}$. If $h=up$ is the polar decomposition of $h$, then $h^{*}h=p^2$ and $hh^*=up^2u^{-1}$. Hence $(C+1)^{-1}\id \leq u^{-1}h^*hu$, and conjugating by $u$ we conclude that 
$$(C+1)^{-1}\id \leq h^*h\leq (C+1)\id.$$
 
Take $h\in H$, since $h^*h-\id$ is compact, $h^*h$ is diagonalizable and has eigenvalues $(\sigma_j)_j\subseteq [(C+1)^{-1},(C+1)]$, hence $\|h^*h-\id\|_p^p=\sum_j|\sigma_j-1|^p\leq C^p$. Note that $\log(h^*h)$ is diagonalizable and has eigenvalues $(\log(\sigma_j))_j$.
  
Now, let $D_1>0$ be a real number such that $|\log(x)|\leq D_1|x-1|$ for all $x\in [(C+1)^{-1},(C+1)]$.
Then 
$$\|\log(h^*h)\|_p^p=\sum_j|\log(\sigma_j)|^p\leq \sum D_1^p|s_j-1|^p \leq D_1^pC^p.$$
Hence, if $\sup_{h\in H}\|h^*h-\id\|_p= C$ then $\sup_{h\in H}\|\log(h^*h)\|_p\leq DC$.

We prove the converse. Assume that $\sup_{h\in H}\|\log(h^*h)\|_p=C< \infty$. Then $\|\log(h^*h)\|\leq\|\log(h^*h)\|_p\leq C$, so that $-C\id\leq\log(h^*h)\leq C\id$, and applying the exponential map $e^{-C}\id\leq h^*h\leq e^C\id$.
Let $D_2>0$ be a real number such that $|\log(x)|\leq D_2|x-1|$ for all $x\in [e^{-C},e^C]$. Then
$$C^p\geq\|\log(h^*h)\|_p^p=\sum_j|\log(\sigma_j)|^p\geq \sum_j D_2^p|s_j-1|^p \geq D_2^p\|h^*h-\id\|_p^p.$$
Hence, if $\sup_{h\in H}\|\log(h^*h)\|_p=C$ then $\sup_{h\in H}\|h^*h-\id\|_p\leq D_2^{-1}C$.
\end{proof}

The next theorem is a generalization of Theorem 4.12 in \cite{ms}.

\begin{thm}\label{unitarization}
If $1<p<\infty$ and $H\subseteq G_{p,ext}$ is a subgroup such that $\sup_{h\in H}\|h^*h-\id\|_p< \infty$, then there exists $s\in P_p$ such that $s^{-1}Hs\subseteq U$.   
\end{thm}

\begin{proof}
By assumption $\sup_{h\in H}\|h^*h-\id\|_p< \infty$, and by Lemma \ref{boundedorbit} this is equivalent to
$$\sup_{h\in H}d_p(h\cdot \id,\id)=\sup_{h\in H}d_p(h^*h,\id)=\sup_{h\in H}\|\log(h^*h)\|_p< \infty.$$
Hence, the orbit $\oo_H(\id)$ of the identity is bounded. The space $P_p$ is $p$-Busemann, hence by Proposition 3.14 in \cite{condelarotonda2} the bounded sets have circumcenters. By the Bruhat-Tits fixed point theorem, see \cite{bruhattits} and \cite{lang}*{Chapter XI, Lemma 3.1 and Theorem 3.2}, the circumcenters of orbits of isometric actions with bounded orbits are fixed points. That is, if $a$ is the circumcenter of $\oo_H(\id)$, then $h\cdot a$ is the circumcenter of $h\cdot\oo_H(\id)=\oo_H(\id)$, so that $h\cdot a=a$ for all $h\in H$. Hence, $a\in P_p$ is a fixed point for the action of $H$ and by Proposition \ref{fixed}, $s=a^{-\frac12}\in P_p$ is a unitarizer of $H$.
\end{proof}

\begin{rem}
If we try to prove the previous theorem using the Ryll-Nardzewski fixed point theorem we have the following difficulty. If we identify $P_p$ with a subset of $\B_p(\h)_s$ using the map $p\mapsto p-\id$ then the action on $P_p$ becomes the affine action on $\B_p(\h)_s$ given by $$h\cdot s=(h^{-1})^*(s+\id)h^{-1}-\id.$$
We assume for simplicity that $H$ is the group of unitaries $U$ and that $p=2$, so that we have to prove that the conjugation action by unitaries on the self-adjoint Hilbert-Schmidt operators has a fixed point. The zero operator is in the weak closure of the orbit of any operator, hence the action is not distal and this key assumption of the Ryll-Nardzewski fixed point theorem is not satisfied. 
\end{rem}

We next show that norms which are Schatten close to Hilbert norms have unitarizable isometry groups.

\begin{lem}\label{cotas}
If $c\geq \id$ satisfies $c-\id\in\B_p$, and $H$ is a group of invertibles such that $h^*h\leq c$ for all $h\in H$, then $H\subseteq G_{p,ext}$ and
$$\|h^*h -\id\|_p\leq 2^{\frac{1}{p}+1} \|c-\id\|_p$$
for all $h\in H$.
\end{lem}

\begin{proof}
Applying the inverse map to $h^*h\leq c$ we get $c^{-1}\leq(h^*h)^{-1}=h^{-1}(h^{-1})^*$. If $h^{-1}$ has polar decomposition $h^{-1}=up$, then $h^{-1}(h^{-1})^*=u(h^{-1})^*h^{-1}u^{-1}$. Since $\id-c\leq c^{-1}-\id$, if in the previous equation we set $h^{-1}$ instead of $h$, then for a unitary $u$ we have $u^{-1}(\id-c)u\leq h^*h -\id$, so that 
$$u^{-1}(\id-c)u\leq h^*h -\id\leq c-\id.$$
Hence $h^*h -\id$ is a compact operator. Since $u^{-1}(\id-c)u\leq 0$ and $c-\id\geq 0$ we get
$$\id-c+u^{-1}(\id-c)u\leq u^{-1}(\id-c)u\leq h^*h -\id\leq c-\id\leq c-\id + u^{-1}(c-\id)u.$$
If we denote $d=c-\id + u^{-1}(c-\id)u$, then $-d\leq h^*h -\id\leq d$, and by Theorem 2.1 in \cite{kit} (see also \cite{bhatia}) 
$$\mu_i(h^*h -\id)\leq \mu_i(d\oplus d)\mbox{  for  }i\in\N,$$
where the $\mu_i$ denote the singular values. This implies that $\|h^*h -\id\|_p\leq 2^{\frac{1}{p}} \|d\|_p$. But $ \|d\|_p= \|c-\id + u^{-1}(c-\id)u\|_p\leq \|c-\id\|_p+\|u^{-1}(c-\id)u\|_p=2\|c-\id\|_p$, hence 
$$\|h^*h -\id\|_p\leq 2^{\frac{1}{p}+1} \|c-\id\|_p\mbox{  for all  }h\in H.$$
Note that $H\subseteq G_{p,ext}$ by Proposition \ref{chargrupoext}.
\end{proof}

\begin{thm}\label{unitisom}
If $\|\cdot\|$ is Schatten close to the Hilbert norm then $\Isom(\|\cdot\|)$ is unitarizable. We can choose a change of variables $s$ such that the norm $\|s^{-1}\cdot\|$ with unitary isometry group is Schatten close to the Hilbert norm.
\end{thm}

\begin{proof}
By Lemma \ref{changenorm} and Remark \ref{remunitar}  we can assume that a change of variables has been made such that for all $h\in \Isom(\|\cdot\|)$ 
$$\langle h^*h\xi,\xi\rangle=\langle h\xi,h\xi\rangle\leq\|h\xi\|^2=\|\xi\|^2\leq\langle c\xi,\xi \rangle.$$
Hence $h^*h\leq c$ with $c\geq \id$ and $c-\id\in\B_p(\h)$. 

By Lemma \ref{cotas}, $H\subseteq G_{p,ext}$ and $\sup_{h\in H}\|h^*h-\id\|_p< \infty$ so that the conditions of Theorem \ref{unitarization} are satisfied, hence $H$ is unitarizable. A fixed point $e\in P_p$ of the action of $H$ corresponds by Proposition \ref{fixed} to a unitarizer $s=e^{-\frac12}$ of $H$. It is easy to check that if $\|\cdot\|$ is Schatten close to the standard Hilbert norm, then the norm $\|e^{-\frac12}\cdot\|$, which has isometries that are unitary operators, is also Schatten close to the standard Hilbert norm.
\end{proof}

\subsection{Natural invariant sets associated to a norm}\label{s2}

Given a norm $\|\cdot\|$ we define two subsets of $P_p$
\begin{align*}
C^-(\|\cdot\|)&=\{a\in P_p:\|\cdot\|_a\leq\|\cdot\|\} \\
C^+(\|\cdot\|)&=\{a\in P_p:\|\cdot\|\leq\|\cdot\|_a\}.
\end{align*}
When the norm is clear from the context we just write $C^-$ and $
C^+$. The proof of the next proposition is straightforward and we omit it.

\begin{prop}\label{intersec}
The sets $C^-$ and $C^+$ are invariant for the action of the isometry group of the norm. The intersection $C^-\cap C^+$ is empty or a singleton. It is a singleton $\{a\}$ if and only if the norm is Hilbertian with $\|\cdot\|=\|\cdot\|_a$. 
\end{prop}

The sets $C^-$ and $C^+$ are related by polar duality in the following sense: given a norm $\|\cdot\|$ compatible with the Hilbert norm its polar dual is the norm $\|\cdot\|^\circ$ given by
$$\|\xi\|^\circ=\sup\{|\langle \xi,\eta\rangle |:\eta\in\h\mbox{  such that  }\|\eta\|\leq 1\}.$$
For Hilbertian norms we have 
$$\|\cdot\|_a^\circ=\|\cdot\|_{a^{-1}}.$$ 

\begin{lem}\label{polarc}
For a norm $\|\cdot\|$ compatible with the Hilbert norm we have $C^+(\|\cdot\|)=C^-(\|\cdot\|^\circ)^{-1}$.
\end{lem}

\begin{proof}
The inequality $\|\cdot\|\leq\|\cdot\|_a$ is equivalent by polar duality $\|\cdot\|_{a^{-1}}=\|\cdot\|_a^\circ\leq\|\cdot\|^\circ$. Observe that
\begin{align*}
C^+(\|\cdot\|)&=\{a\in P_p:\|\cdot\|\leq\|\cdot\|_a\}=\{a\in P_p:\|\cdot\|_{a^{-1}}\leq\|\cdot\|^\circ\}\\
&=\{a\in P_p:\|\cdot\|_{a}\leq\|\cdot\|^\circ\}^{-1}=C^-(\|\cdot\|^\circ)^{-1}.
\end{align*}
\end{proof}

\begin{thm}\label{convexity}
The sets $C^-$ and $C^+$ are closed and geodesically convex in $P_p$.
\end{thm}

\begin{proof}
We prove that $C^-$ is closed, the proof that $C^+$ is closed is analogous. Let $(a_n)_n$ be a sequence in $C^-$ such that $a_n\to a$ in $P_p$. We claim that $\|a_n-a\|\to 0$. By the exponential metric increasing property, see Corollary 2.3 in \cite{conde} $\|\log(a_n)-\log(a)\|_p\leq d_p(a_n,a)$. Since $\|\log(a_n)-\log(a)\|\leq \|\log(a_n)-\log(a)\|_p$ the claim now follows from the continuity of the exponential map. Hence, for $\xi\in\h$ we have $\langle a_n\xi,\xi\rangle \leq \|\xi\|^2$ and $\langle a_n\xi,\xi\rangle \to \langle a\xi,\xi\rangle$, so that $\langle a\xi,\xi\rangle\leq \|\xi\|^2$.

The geodesic convexity of $C^-$ follows from the geodesic convexity of positive functionals restricted to $P_p$, see Theorem 3 in \cite{cpr3}. For $\xi\in\h$ take the pure state $\phi_\xi(a)=\langle a\xi,\xi\rangle$ which is positive, hence geodesically convex. If $a,b\in C^-$ then for $t\in [0,1]$ 
$$\phi_\xi( \gamma_{a,b}(t))\leq (1-t)\phi_\xi(a)+ t\phi_\xi(b)\leq \|\xi\|^2.$$ 
To prove the geodesic convexity of $C^+$ we use Lemma \ref{polarc} which states that $C^+(\|\cdot\|)=C^-(\|\cdot\|^\circ)^{-1}$. Since the inversion map $a\mapsto a^{-1}$ is an isometry by Proposition 2.6 in \cite{conde}, it maps geodesics to geodesics. This follows also from the fact that inversion is a Cartan symmetry in $P_p$. From the geodesic convexity of $C^-({\|\cdot\|^\circ})$ we conclude the geodesic convexity of $C^+(\|\cdot\|)=C^-(\|\cdot\|^\circ)^{-1}$.
\end{proof}

\begin{rem}
As recalled in the preliminaries, for $1<p<\infty$ the metric spaces $(P_p,d_p)$ satisfy a generalized semi-parallelogram law, hence there are well defined projections to closed convex sets. That is, given a point $a$ there is a unique $b$ in $C$ which minimizes the distance to $a$. This could be applied to the sets $C^-$ and $C^+$.
\end{rem}

\begin{rem}
Propositions \ref{intersec}, Lemma \ref{polarc} and Theorem \ref{convexity} are valid for the space of positive invertible operators $P$ on the Hilbert space $\h$. 
\end{rem}

\subsection{The unique fixed point property}\label{s3}

\begin{prop}\label{irred}
For $1<p<\infty$, a group of unitaries $H\subseteq U$ acting on $P_p$ has $\id$ as its unique fixed point if and only if $H$ leaves no finite dimensional space invariant. 
\end{prop}

\begin{proof}
Note that the set of fixed points of the action of a group of unitaries $H$ is $H'\cap P_p$, where $H'$ is the commutant of $H$.
The set of fixed points of the action of $H$ is $H'\cap P_p$. If there is an $a\in H'\cap P_p$ such that $a\neq \id$, then the compact perturbation of the identity $a$ has a positive  eigenvalue which is not $1$, since $a\neq \id$. The eigenspace corresponding to this eigenvalue is finite dimensional and invariant for $H$. Conversely, if $q$ is a projection onto a finite dimensional subspace invariant under $H$, then $\id+q\in H'\cap P_p$, and $\id+q\neq \id$.  
\end{proof}

\begin{rem}\label{invirred}
Note that if a group of invertibles $H$ leaves invariant a subspace $\h_1$, then for an invertible $g$ we have that $gHg^{-1}$ leaves $g\h_1$ invariant. Therefore using the change of variables and translation properties of Section \ref{basicprop} we conclude that a group $H\subseteq G_{p,ext}$ acting on $P_p$ has a unique fixed point if and only if $H$ leaves no finite dimensional space invariant.
\end{rem}

\begin{ex}\label{shift}
Finite permutations and shifts are groups of unitaries with $\id$ as the unique fixed point. More precisely, given an orthonormal basis $(\xi_n)_{n\in \Z}$ define for an integer $m\in\Z$ the unitary $u_m$ which maps $\xi_n$ to $\xi_{n+m}$, and for a finite permutation $\pi$ of $\Z$ let $u_\pi$ be the unitary which maps $\xi_n$ to $\xi_{\pi(n)}$. We can assume that $p=2$. Let $a\in P_2$ be an element invariant under each of these groups. If its matrix is given by $(a_{i,j})$ then an easy argument shows that all diagonals have equal entries, that is, $a_{i,j}=a_{i+n,j+n}$ for all $i,j,n\in\Z$. Note that in this case $a\in \id+\B_2(\h)$ can only happen if $a=\id$, since the Hilbert-Schmidt norm is the $l^2$ norm of the matrix entries. The fact that permutations have $\id$ as its unique fixed point does not hold in the finite dimensional context. 
\end{ex}

\begin{rem}
It is immediate that if the isometry group $H$ acts with dense orbits on the sphere of the norm, then it leaves no finite dimensional subspace invariant. 
\end{rem}

\begin{thm}\label{teouniquefixed}
Let $\|\cdot\|$ be a norm which is Schatten close to a Hilbert norm. If the group of isometries $H$ leaves no finite dimensional space invariant, then $\|\cdot\|$ is the Hilbert norm.
\end{thm}

\begin{proof}
By Theorem \ref{unitisom} and Remark \ref{invirred} we can assume that a change of variables has been made such that the group of isometries is a group $H$ of unitaries, and the norm is Schatten close to the standard Hilbert norm. Since it is Schatten close to the standard Hilbert norm the sets $C^-$ and $C^+$ are not empty. The sets $C^-$ and $C^+$ are also closed, geodesically convex and $H$-invariant by Proposition \ref{intersec} and Theorem \ref{convexity}.   

We claim that if a set $C\subseteq P_p$ is non-empty, closed, geodesically convex and $H$-invariant then $\id\in C$. The set $C$ is a complete $p$-Busemann space because it is a closed and geodesically convex subspace of a $p$-Busemann space. Since it is also $H$-invariant we can apply the Bruhat-Tits fixed point theorem for the action of $H$ on $C$ to get a fixed point $c\in C$. By Proposition \ref{irred} the action has $\id$ as its unique fixed point, and $c$ is a fixed point so we must have $\id=c\in C$. From this claim we conclude that $\id\in C^-$ and $\id\in C^+$. This is equivalent to 
$$\|\cdot\|_{\id}\leq\|\cdot\|\leq\|\cdot\|_{\id},$$
so that $\|\cdot\|=\|\cdot\|_{\id}$.
\end{proof}

The next corollary follows from Theorem \ref{teouniquefixed}. It can be proved with more elementary techniques.

\begin{cor}
If $\|\cdot\|$ is a norm invariant under shifts or finite permutations and which is Schatten close to the standard norm, then it is the standard norm.
\end{cor}

\noindent

\begin{bibdiv}
\begin{biblist}

\bib{kit}{article}{author={W. Audeh}, author={F. Kittaneh}, title={Singular value inequalities for compact operators}, journal={Linear Algebra Appl.}, number={437}, date={2012}, pages={no. 10, 2516--2522}}

\bib{banach}{book}{author={S. Banach}, title={Th\'eorie des op\'erations lin\'eaires}, series={Éditions Jacques Gabay, Sceaux}, date={1993}}

\bib{bhatia}{article}{author={R. Bhatia}, author={F. Kittaneh}, title={The matrix arithmetic-geometric mean inequality revisited}, journal={Linear Algebra Appl.}, number={428}, date={2008}, pages={no. 8-9, 2177--2191}}

\bib{bruhattits}{article}{author={F. Bruhat}, author={J. Tits}, title={Groupes r\'eductifs sur un corps local, I. Donn\'ees radicielles valu\'ees}, journal={Inst. Hautes \'Etudes Sci. Publ. Math.}, volume={41}, date={1972}, pages={5-252}}

\bib{conde}{article}{author={C. Conde}, title={Nonpositive curvature in $p$-Schatten class}, journal={J. Math. Anal. Appl.}, number={356}, date={2009}, pages={2, 664--673}}

\bib{condelarotonda2}{article}{author={C. Conde}, author={G. Larotonda}, title={Manifolds of semi-negative curvature}, journal={Proc. Lond. Math. Soc.}, number={3}, date={2010}, pages={no. 3, 670--704}}

\bib{cpr3}{article}{author={G. Corach},author={H. Porta},author={L. Recht},title={Geodesics and operator means in the space of positive operators}, journal={Internat. J. Math.}, number={4}, date={1993}, pages={no. 2, 193--202}}

\bib{fr}{article}{author={V. Ferenczi},author={C. Rosendal},title={On isometry groups and maximal symmetry}, journal={Duke Math. J.}, number={10}, date={2013}, pages={no. 2, 1771--1831}}

\bib{lang}{book}{author={S. Lang}, title={Fundamentals of Differential Geometry}, series={Graduate Texts in Mathematics}, publisher={Springer-Verlag, New York}, date={1999}}

\bib{ms}{article}{author={M. Miglioli}, author={P. Schlicht}, title={Geometric aspects of similarity problems}, journal={Int. Math. Res. Not. IMRN}, number={23}, date={2018}, pages={7171--7197}}

\bib{neeb}{article}{author={K.-H. Neeb}, title={A Cartan-Hadamard Theorem for Banach-Finsler manifolds}, journal={Geom. Dedicata}, number={95}, date={2002}, pages={115-156}}

\end{biblist}
\end{bibdiv}
\end{document}